\documentclass[11pt,a4paper]{article}
\usepackage{listings}
\usepackage{a4wide}
\usepackage{url}
\usepackage{amssymb} %Box
\usepackage{amsmath} %Operator
\usepackage{amsthm} %proof
\usepackage[numbers]{natbib}
\usepackage[title]{appendix}

\newtheorem{theorem}{Theorem}

\newtheorem{lemma}[theorem]{Lemma}
\newtheorem{corollary}[theorem]{Corollary}

\newtheorem{conjecture}[theorem]{Conjecture}

\newtheorem{question}[theorem]{Question}

\title{Generalised Majority Colourings of Digraphs}

\author{Ant\'onio Gir\~ao\thanks{\,Department of Pure Mathematics and Mathematical Statistics, University of Cambridge, Wilberforce Road, Cambridge CB3 0WB, UK;
    \texttt{A.Girao@dpmms.cam.ac.uk}.}
   \and Teeradej Kittipassorn\thanks{\,Departamento de Matem\'atica, Pontif\'icia Universidade Cat\'olica do Rio de Janeiro (PUC-Rio), Rua Marqu\^es de S\~ao Vicente 225, G\'avea, Rio de Janeiro, RJ 22451-900, Brazil; \texttt{ping41@mat.puc-rio.br}.}
  \and Kamil Popielarz\thanks{\,Department of Mathematics, University Of Memphis,
      Memphis, TN 38152, USA; \texttt{kamil.popielarz@\-gmail.com} }}

\begin{document}

\maketitle
\begin{abstract}
    We almost completely solve a number of problems related to a concept called majority colouring recently studied by Kreutzer, Oum, Seymour, van der Zypen and Wood. They raised the problem of determining, for a natural number $k$, the smallest number $m=m(k)$ such that every digraph can be coloured with $m$ colours where each vertex has the same colour as at most a $1/k$ proportion of its out-neighbours. We show that $m(k)\in\{2k-1,2k\}$. We also prove a result supporting the conjecture that $m(2)=3$. Moreover, we prove similar results for a more general concept called majority choosability.
\end{abstract}

For a natural number $k\geq 2$, a \emph{$\frac{1}{k}$-majority colouring} of a digraph is a colouring of the vertices such that each vertex receives the same colour as at most a $1/k$ proportion of its out-neighbours. We say that a digraph $D$ is \emph{$\frac{1}{k}$-majority $m$-colourable} if there exists a $\frac{1}{k}$-majority colouring of $D$ using $m$ colours. The following natural question was recently raised by Kreutzer, Oum, Seymour, van der Zypen and Wood~\cite{Kreutzer}.

\begin{question}
\label{que:main}
Given $k\geq 2$, determine the smallest number $m=m(k)$ such that every digraph is $\frac{1}{k}$-majority $m$-colourable.
\end{question}

In particular, they asked whether $m(k)=O(k)$. Let us first observe that $m(k)\geq 2k-1$. Consider a tournament on $2k-1$ vertices where every vertex has out-degree $k-1$. Any $\frac{1}{k}$-majority colouring of this tournament must be a proper vertex-colouring, and hence it needs at least $2k-1$ colours. Conversely, we prove that $m(k)\leq 2k$.

\begin{theorem}
\label{thm1}
Every digraph is $\frac{1}{k}$-majority $2k$-colourable for all $k\geq 2$.
\end{theorem}

This is an immediate consequence of a result of Keith Ball (see~\cite{Bourgain}) about partitions of matrices. We shall use a slightly more general version proved by Alon~\cite{Alon}.

\begin{lemma}
\label{lem1}
    Let $A = (a_{ij})$ be an $n\times n$ real matrix where $a_{ii} = 0$ for all $i$, $a_{ij} \ge 0$ for all $i \neq j$, and $\sum_{j}a_{ij} \le 1$ for all $i$. Then, for every $t$ and all positive reals $c_{1}, \dots, c_{t}$ whose sum is $1$, there is a partition of $\{1, 2, \dots, n\}$ into pairwise disjoint sets $S_{1}, S_{2}, \dots, S_{t}$, such that for every $r$ and every $i \in S_{r}$, we have $\sum_{j \in S_{r} }a_{ij} \le 2c_{r}$.
\end{lemma}

\begin{proof}[Proof of Theorem~\ref{thm1}] 
    Let $D$ be a digraph on $n$ vertices with vertex set $\left\{ v_{1}, v_{2}, \dots, v_{n} \right\}$ and write $d^{+}(v_{i})$ for the out-degree of $v_i$.
    Let $A = (a_{ij})$ be an $n\times n$ matrix where $a_{ij} = \frac{1}{d^{+}(v_{i})}$ if there is a directed edge from $v_{i}$ to $v_{j}$ and $a_{ij} = 0$ otherwise.
    We apply Lemma~\ref{lem1} with $t = 2k$ and $c_{i} = \frac{1}{2k}$ for $1 \le i \le 2k$ obtaining a partition of $\left\{ 1, 2, \dots, n \right\}$ into sets $S_{1}, S_{2}, \dots, S_{2k}$, such that for every $r$ and every $i \in S_{r}$, $\sum_{j \in S_{r}} a_{ij} \le \frac{1}{k}$. Equivalently, the number of out-neighbours of $v_i$ that have the same colour as $v_i$ is at most $\frac{d^{+}(v_{i})}{k}$ where the colouring of $D$ is defined by the partition $S_1\cup S_2\cup\dots\cup S_{2k}$.
\end{proof}

Question~\ref{que:main} has now been reduced to whether $m(k)$ is $2k-1$ or $2k$.
\begin{question}
\label{que:remaining}
Is every digraph $\frac{1}{k}$-majority $(2k-1)$-colourable?
\end{question}

Surprisingly, this is open even for $k=2$. Kreutzer, Oum, Seymour, van der Zypen and Wood~\cite{Kreutzer} gave an elegant argument showing that every digraph is $\frac{1}{2}$-majority $4$-colourable and they conjectured that $m(2)=3$.

\begin{conjecture}
\label{conj:2}
    Every digraph is $\frac{1}{2}$-majority $3$-colourable.
\end{conjecture}

We provide evidence for this conjecture by proving that tournaments are \emph{almost} $\frac{1}{2}$-majority $3$-colourable.

\begin{theorem}
\label{almost_theorem}
Every tournament can be $3$-coloured in such a way that all but at most $205$ vertices receive the same colour as at most half of their out-neighbours. 
\end{theorem}
\begin{proof}
The proof relies on an observation that in a tournament $T$, the set $S_{i}=\{x\in V(T): 2^{i-1} \leq d^{+}(x) < 2^{i}\}$ has size at most $2^{i+1}$. Indeed, the sum of the out-degrees of the vertices of $S_i$ is at least $\binom{|S_i|}{2}$, the number of edges inside $S_i$. On the other hand, this sum is at most $(2^{i}-1)|S_i|$ by the definition of $S_i$. Therefore, $\binom{|S_i|}{2} \le (2^{i}-1)|S_i|$ and hence, $|S_i| \le 2^{i+1}-1$.

We proceed by randomly assigning one of three colours to each vertex independently with probability $1/3$. 
Given a vertex $x$, let $B_{x}$ be the number of out-neighbours of $x$ which receive the same colour as $x$.
We say that $x$ is \emph{bad} if $B_{x} > d^{+}(x)/2$. Trivially $\mathbb{E}(B_x) = d^{+}(x)/3$, and hence by a Chernoff-type bound, it follows that, for $x\in S_i$,
\begin{align*}
\label{eq1}
\mathbb{P} (x \text{ is bad}) &= \mathbb{P}(B_{x} > d^{+}(x)/2) = \mathbb{P}\left(B_{x} >  (1 + 1/2) \mathbb{E}(B(x))\right) \\
& \leq \exp \left(-\frac{(1/2)^2}{3}\mathbb{E}(B_{x})\right) = \exp(-d^{+}(x)/36) \le \exp(-2^{i-1}/36).
\end{align*}
Notice that if $i \ge 11$ then $\mathbb{P}(x \text{ is bad}) \le 2^{-(2i-7)}$. Let $X$ denote the total number of bad vertices. Since the vertices of out-degree $0$ cannot be bad,
\begin{align*}
    \mathbb{E}(X) &= \sum_{i\ge 1}\sum_{x\in S_i} \mathbb{P} (x \text{ is bad}) \le \sum_{i = 1}^{10} 2^{i+1} \exp(-2^{i-1}/36)+ \sum_{i \ge 11} 2^{i+1} 2^{-(2i-7)}  \\
    & \le 205 + \sum_{i \ge 11}2^{-i + 8} = 205 + \frac{1}{4} < 206.
\end{align*}
Hence, there is a $3$-colouring such that all but at most $205$ vertices receive the same colour as at most half of their out-neighbours. 
\end{proof}

Observe also that the same argument proves a special case of Conjecture~\ref{conj:2}.

\begin{theorem}\label{mindeg}
    Every tournament with minimum out-degree at least $2^{10}$ is $\frac{1}{2}$-majority $3$-colour\-able.
\end{theorem}

We remark that Theorem~\ref{almost_theorem} can be strengthened ($205$ can be replaced by $7$) by solving a linear programming problem. Recall that the expected number of bad vertices of out-degree at least $1024$ is at most $1/4$. We shall use linear programming to show that the expected number of bad vertices of out-degree less than $1024$ is less than $7.75$.  Let $V_{i}$ be the set of vertices of out-degree $i$ for $i \in \left\{ 1,2, \dots, 1023 \right\}$ and note that the expected number of bad vertices of out-degree at most $1023$ is $f(v_{1},v_2, \dots, v_{1023}) = \sum_{i=1}^{1023}v_{i}p_{i}$ where $v_{i} = |V_{i}|$ and $p_{i} = \sum_{j = \lceil \frac{i+1}{2} \rceil}^{i}\binom{i}{j}(1/3)^{j}(2/3)^{i-j}$.
As before, observe that the number of vertices of degree at most $i$ is at most $2i+1$, and therefore, $\sum_{j=1}^{i} v_{i} \le 2i+1$, leading to the following linear program.
\begin{align*}
    \text{Maximize: } & f(v_{1},v_2, \dots, v_{1023}) \\
    \text{Subject to: } & \sum_{j=1}^{i} v_{j} \le 2i+1, \text{ for } i \in \{1,2, \dots,  1023\} \\
    \text{Subject to: } & v_{i} \ge 0, \text{ for } i \in \{1,2,  \dots, 1023\} .
\end{align*}

See Appendix~\ref{app} for the source code. Similarly, we can replace $2^{10}$ in Theorem~\ref{mindeg} by $55$, by using the same linear program to show that the expected number of bad vertices of out-degree in $[55,1023]$ is less than $3/4$.

Let us now change direction to a more general concept of majority choosability. A digraph is \emph{$\frac{1}{k}$-majority $m$-choosable} if for any assignment of lists of $m$ colours to the vertices, there exists a $\frac{1}{k}$-majority colouring where each vertex gets a colour from its list. In particular, a $\frac{1}{k}$-majority $m$-choosable digraph is $\frac{1}{k}$-majority $m$-colourable. Kreutzer, Oum, Seymour, van der Zypen and Wood~\cite{Kreutzer} asked whether there exists a finite number $m$ such that every digraph is $\frac{1}{2}$-majority $m$-choosable. Anholcer, Bosek and Grytczuk~\cite{Anholcer} showed that the statement holds with $m=4$. We generalise their result as follows.

\begin{theorem}
\label{thm2}
Every digraph is $\frac{1}{k}$-majority $2k$-choosable for all $k\geq 2$.
\end{theorem}

Theorem~\ref{thm2} was independently proved by Fiachra Knox and Robert \v{S}\'amal~\cite{Knox}. We prove Theorem~\ref{thm2} using a slight modification of Lemma~\ref{lem1} whose proof is very similar to that of Lemma~\ref{lem1}.

\begin{lemma}
\label{lem2}
    Let $A = (a_{ij})$ be an $n\times n$ real matrix where $a_{ii} = 0$ for all $i$, $a_{ij} \ge 0$ for all $i \neq j$, and $\sum_{j}a_{ij} \le 1$ for all $i$. 
    Then, for every $m$ and subsets $L_{1}, L_{2}, \dots, L_{n} \subset \mathbb{N}$ of size $m$, there is a function $f:\left\{ 1, 2, \dots, n \right\} \rightarrow \mathbb{N}$ such that, for every $i$, $f(i) \in L_{i}$ and $\sum_{j \in f^{-1}(r)} a_{ij} \le \frac{2}{m}$ where $r=f(i)$.
\end{lemma}
\begin{proof}
    By increasing some of the numbers $a_{ij}$, if needed, we may assume that $\sum_{j}a_{ij} = 1$ for all $i$. We may also assume, by an obvious continuity argument, that $a_{ij} > 0$ for all $i \neq j$. Thus, by the Perron--Frobenius Theorem, $1$ is the largest eigenvalue of $A$ with right eigenvector $\left( 1, 1,  \dots, 1 \right)$ and left eigenvector $\left( u_1, u_2,  \dots, u_n \right)$ in which all entries are positive. It follows that $\sum_{i}u_{i}a_{ij} = u_{j}$. Define $b_{ij} = u_{i}a_{ij}$, then $\sum_{i}b_{ij} = u_{j}$ and $\sum_{j}b_{ij} = u_{i} \left( \sum_{j}a_{ij} \right) = u_{i}$.

    Let $f: \left\{ 1,2,\dots,n \right\} \rightarrow \mathbb{N}$ be a function such that $f(i) \in L_{i}$ and $f$ minimises the sum $\sum_{r \in \mathbb{N}} \sum_{i, j \in f^{-1}(r)}b_{ij}$. By minimality, the value of the sum will not decrease if we change $f(i)$ from $r$ to $l$ where $l \in L_{i}$. Therefore, for any $i \in f^{-1}(r)$ and $l \in L_{i}$, we have
    \[
        \sum_{j \in f^{-1}(r)}(b_{ij} + b_{ji}) \le \sum_{j \in f^{-1}(l)}(b_{ij} + b_{ji}).
    \]
    Summing over all $l \in L_{i}$, we conclude that
		\[
            m\sum_{j \in f^{-1}(r)}(b_{ij} + b_{ji}) \le \sum_{j \in f^{-1}(L_i)}(b_{ij} + b_{ji}) \le \sum_{j = 1}^{n}(b_{ij} + b_{ji}) = 2u_{i}.
		\]
    Hence, $\sum_{j \in f^{-1}(r)}u_{i}a_{ij} = \sum_{j \in f^{-1}(r)} b_{ij} \le \sum_{j \in f^{-1}(r)}(b_{ij} + b_{ji}) \le \frac{2u_{i}}{m}$. Dividing by $u_{i}$, the desired result follows.
\end{proof}

\begin{proof}[Proof of Theorem~\ref{thm2}]
    The proof is the same as that of Theorem~\ref{thm1}, using Lemma~\ref{lem2} instead of Lemma~\ref{lem1}.
\end{proof}

In fact, the same statement also holds when the size of the lists is odd.

\begin{corollary}
Every digraph is $\frac{2}{m}$-majority $m$-choosable for all $m\geq 2$.
\end{corollary}

This statement generalises a result of Anholcer, Bosek and Grytczuk~\cite{Anholcer} where they prove the case $m=3$ which says that, given a digraph with colour lists of size three assigned to the vertices, there is a colouring from these lists such that each vertex has the same colour as at most two thirds of its out-neighbours. 

We have established that the $\frac{1}{k}$-majority choosability number is either $2k-1$ or $2k$. Let us end this note with an analogue of Question~\ref{que:remaining}.

\begin{question}
    Is every digraph $\frac{1}{k}$-majority $(2k-1)$-choosable?
\end{question}

\bibliographystyle{siam}
\bibliography{majcol}

\begin{appendices}
\section{Linear program}\label{app}
We use the toolkit \cite{GLPK} to solve the linear program with the following source code:
\vspace{20pt}
\begin{lstlisting}
param N := 1024;

param comb 'n choose k' {n in 0..N, k in 0..n} :=
    if k = 0 or k = n then 1 else comb[n-1,k-1] + comb[n-1,k];
param prob 'probability' {n in 0..N} :=
    sum{k in (floor(n/2)+1)..n} comb[n,k]*((1/3)^k)*((2/3)^(n-k));

var x{1..N}, integer, >= 0;

subject to constraint{i in 1..N}: sum{j in 1..i} x[j] <= 2*i+1;

maximize expectation: sum{i in 1..N} x[i]*prob[i];

solve;

end;
\end{lstlisting}
\end{appendices}

\end{document}